\documentclass[11pt]{article}

\usepackage{amsthm}
\usepackage{amssymb}
\usepackage{tikz}
\usepackage{verbatim}
\usepackage[numbers]{natbib}

\usepackage{rotating}
\usepackage{url}
\usepackage{float}
\usepackage{graphicx}
\usepackage{color}
\usepackage{latexsym}
\usepackage{amsfonts}
\usepackage{amsmath}
\usepackage{psfrag}
\usepackage{mathrsfs}
\usepackage{algpseudocode}
\usepackage[ruled]{algorithm}
\usepackage{caption}

\newtheorem{proposition}{Proposition}

\newtheorem{theorem}{Theorem}

\newtheorem{lemma}{Lemma}

\usepackage{mathrsfs}

\def\R{\mathbb{R}}
\def\Z{\mathbb{Z}}
\def\N{\mathbb{N}}

\def\eps{\epsilon}
\def\xstar{x^{\star}}

 \setbox\strutbox=\hbox{\vrule height7pt depth2pt width0pt}

\def\square{{\setbox0=\hbox{X}\hbox to \ht0{\vrule\hss\vbox to \ht0{
  \hrule width \ht0\vfil\hrule width \ht0}\vrule}}}

\begin{document}
\pagestyle{empty}

\vskip 2cm \begin{center}{\huge A Frank-Wolfe Based Branch-and-Bound Algorithm\\[1ex] for Mean-Risk Optimization}\end{center}
\par\bigskip
\centerline{\Large C. Buchheim$^\dag$, M. De Santis$^\ddag$, F. Rinaldi$^*$, L. Trieu$^\dag$}
\par\bigskip\bigskip

\centerline{$^{\dag}$Fakult\"at f\"ur Mathematik } \centerline{TU
    Dortmund } \centerline{Vogelpothsweg 87 - 44227 Dortmund - Germany}
\par\medskip
\centerline{$^{\ddag}$Institut f\"ur Mathematik}  \centerline{Alpen-Adria-Universit\"at Klagenfurt}
           \centerline{Universit\"atsstrasse 65-67, 9020 Klagenfurt - Austria}
\par\medskip
\centerline{ $^{*}$Dipartimento di Matematica}\centerline{Universit\`a  di Padova }\centerline{Via Trieste, 63 - 35121 Padova -
Italy}
\par\medskip
\centerline{e-mail (Buchheim): christoph.buchheim@tu-dortmund.de}
\centerline{e-mail (De Santis): marianna.desantis@aau.at}
\centerline{e-mail (Rinaldi): rinaldi@math.unipd.it}
\centerline{e-mail (Trieu): long.trieu@math.tu-dortmund.de}

\par\bigskip\noindent {\small \centerline{\bf Abstract} 
We present an
  exact algorithm for mean-risk optimization subject to a budget
  constraint, where decision variables may be continuous or integer.
  The risk is measured by the covariance matrix and weighted by an
  arbitrary monotone function, which allows to model risk-aversion in
  a very individual way. We address this class of convex mixed-integer
  minimization problems by designing a branch-and-bound algorithm, 
  where at each node, the continuous relaxation is solved 
  by a non-monotone Frank-Wolfe type algorithm with away-steps. 
  Experimental results on portfolio optimization problems show that our approach 
  can outperform the MISOCP solver of CPLEX 12.6 for instances where a linear 
  risk-weighting function is considered.

\begin{quote}

\end{quote} }

\pagestyle{plain} \setcounter{page}{1}


\section{Introduction}

We consider mixed-integer knapsack problems of the form
\begin{equation*}
      \begin{array}{l l}
        \max & c^\top y\\[1ex]
        \textnormal{~s.t. }& a^\top y\leq b\\
        & y\ge 0 \\
        & y_i\in \Z~~\forall{ i\in I},
      \end{array}
\end{equation*}
where~$y\in\R^n$ is the vector of non-negative decision variables,
the index set $I\subseteq\{1,\dots,n\}$ specifies which variables have
to take integer values. In many practical applications, the objective
function coefficients~$c\in\R^n$ are uncertain, while~$a\in\R_+^n$
and~$b\in\R_+$ are known precisely. E.g., in portfolio optimization
problems, the current prices~$a$ and the budget~$b$ are given, but the
returns~$c$ are unknown at the time of investment. The robust
optimization approach tries to address such uncertainty by considering
worst-case optimal solutions, where the worst-case is taken over a
specified set of probable scenarios called the \emph{uncertainty
  set}~$U$ of the problem. Formally, we thus obtain the problem
\begin{equation}\label{prob:robust}
      \begin{array}{l l}
        \max & {\displaystyle\min_{c\in U}}~c^\top y\\[2ex]
        \textnormal{~s.t. }& a^\top y\leq b\\
        & y\ge 0 \\
        & y_i\in \Z~~\forall{ i\in I}.
      \end{array}
\end{equation}
The coefficients~$c_i$ may also be interpreted as random variables.
Assuming a multivariate normal distribution, a natural choice for the
set~$U$ is an ellipsoid defined by the means~$r\in\R^n$ and a
positive definite covariance matrix~$M\in\R^{n\times n}$ of~$c$. 
In this case,
Problem~\eqref{prob:robust} turns out to be equivalent (see, e.g.,~\cite{bertsimas2004}) to the
non-linear knapsack problem
\begin{equation}\label{prob:robust2}
      \begin{array}{l l}
        \max & r^\top y-\Omega\sqrt{y^{\top}M y}\\[1ex]
        \textnormal{~s.t. }& a^\top y\leq b\\
        & y\ge 0 \\
        & y_i\in \Z~~\forall{ i\in I},
      \end{array}
\end{equation}
where the factor~$\Omega\in \R$ corresponds to the chosen confidence
level. It can be used to balance the mean and the risk in the
objective function and hence to model the risk-aversion of the
user. Ellipsoidal uncertainty sets have been widely considered in
robust optimization~\cite{atamtuerk2008,baumann2014, bertsimas2004}.

In fact, mean-risk models such as~\eqref{prob:robust2} have been
studied intensively in portfolio optimization, since Markowitz
addressed them in his seminal paper dating back to
1952~\citep{markowitz1952}. Originally, the risk term was often given
as~$y^{\top}M y$ instead of~$\sqrt{y^{\top}M y}$, which generally
leads to a different optimal balance between mean and risk.  In our
approach, we allow to describe the weight of the risk by any convex,
differentiable and non-decreasing function
$h:\mathbb{R_+}\rightarrow\mathbb{R}$. Typical choices for the
function $h$ could be $h(t)=\Omega t$, yielding~\eqref{prob:robust2},
or $h(t)=\Omega t^2$, which gives a convex MIQP problem. However, it
may also be a reasonable choice to neglect small risks while
trying to avoid a large risk as far as possible, this could be
modeled by an exponential function
\[h(t)=\begin{cases}0 &\ t\leq\gamma \\ \exp(t-\gamma)-(t-\gamma+1) &\ t>\gamma\;.\end{cases}\]
In summary, our aim is to compute exact solutions for problems of the form
\begin{equation}\label{prob:gen_cap_bud}
      \begin{array}{l l}
        \max & r^\top y-h(\sqrt{y^{\top}M y})\\[1ex]
        \textnormal{~s.t. }& a^\top y\leq b\\
        & y\ge 0 \\
        & y_i\in \Z~~\forall{ i\in I}.
      \end{array}
\end{equation}

\subsection{Our contribution}
The main contribution of this paper is an exact algorithm to solve
Problem~\eqref{prob:gen_cap_bud}, i.e.  a class of convex nonlinear
mixed-integer programming problems.  We propose a branch-and-bound
method that suitably combines a Frank-Wolfe like
algorithm~\cite{frank1956algorithm} with a branching strategy already
succesfully used in the context of mixed-integer programming problems
(see~\cite{buchheim2016p} and references therein).

Our approach for solving the continuous relaxation in
each subproblem (i.e. the problem obtained by removing the integrality
constraints) exploits the simple structure of the feasible set
of~\eqref{prob:gen_cap_bud} as well as the specific structure of the
objective function. It uses away-steps as proposed
by~\citet{Guelat1982} as well as a non-monotone line search.

Our motivation to choose a Frank-Wolfe like method is twofold.  On the
one hand, the algorithm, at each iteration, gives a valid dual bound
for the original mixed-integer nonlinear programming problem, thus
enabling fast pruning of the nodes in the branch-and-bound tree. On
the other hand, the running time per iteration is very low, because
the computation of the descent direction and the update of the
objective function can be performed in an efficient way, as it will be
further explained in the next sections.  These two properties, along
with the possibility of using warmstarts, are the key to a fast
enumeration of the nodes in the branch-and-bound algorithm we have
designed.

\subsection{Organization of the paper}
The remaining sections of the paper are organized as follows. 
In Section~\ref{sec:FW} we describe a modified Frank-Wolfe method 
to efficiently compute the dual bounds for the node relaxations. 
The section also includes an in-depth convergence analysis of the algorithm. 
In Section~\ref{bandb} we shortly explain the main ideas of our branch-and-bound algorithm, including 
the branching strategy, upper and lower bound computations and several effective warmstart strategies 
to accelerate the dual bound computation. In Section~\ref{sec:exp} we test our algorithm on 
real-world instances. We show computational results and compare the performances of our algorithm and of~\texttt{CPLEX 12.6} 
for different risk-weighting functions $h$.  Finally, in Section~\ref{sec:concl} we summarize the results and give some conclusions.

\section{A modified version of the Frank-Wolfe method for the fast
  computation of valid dual bounds}
\label{sec:FW}

A continuous convex relaxation of (the minimization version of)
Problem~\eqref{prob:gen_cap_bud}, simply obtained by removing the
integrality constraints in the original formulation, is the following:
\begin{equation}\label{ContRel}
      \begin{array}{l l}
        \min & h\Big(\sqrt{y^\top M y}\Big) - r^\top y\\[1ex]
        \textnormal{~s.t. }& a^\top y\leq b\\
        & y\ge 0\;.
      \end{array}
\end{equation}
By the transformation $y_i=\frac{b}{a_i} x_i$, Problem~\eqref{ContRel} becomes
\begin{equation}\label{ContRel0}
      \begin{array}{l l}
        \min & f(x) = h\Big(\sqrt{x^\top Q x}\Big) - \mu^\top x\\[1ex]
        \textnormal{~s.t. }& \mathbf{1}^\top x \leq 1\\
        & x\ge 0
      \end{array}
    \end{equation}
where $Q_{ij} = \frac {b^2}{a_ia_j}M_{ij}$, $\mu_i = \frac{b}{a_i}
r_i$ and $\mathbf{1} = (1,\ldots, 1)^\top$ is the $n$-dimensional vector with all entries equal to one.  
For the following, let $S = \{ x\in \R^n: \mathbf{1}^\top x\le 1, x\ge
0\}$ denote the feasible set of~\eqref{ContRel0}.

In this section, we consider the Frank-Wolfe
algorithm with away-steps proposed by~\citet{Guelat1982}, and define
a non-monotone version for
solving Problem~\eqref{ContRel0}.  We also analyze its convergence
properties. This algorithm is then embedded into our branch-and-bound
framework. 

The original method described in~\citep{Guelat1982} uses an exact
line search to determine, at a given iteration, the stepsize along the
descent direction that yields the new iterate. When the exact line
search is too expensive (i.e.\ too many objective function and
gradient evaluations are required), different rules can be used for
the stepsize calculation; see e.g.~\citep{freund2014new}. In
particular, inexact line search methods can be applied to calculate
the stepsize~\citep{dunn1980convergence}, such as the Armijo or
Goldstein line search rules. Typical line search algorithms try a
sequence of candidate values for the stepsize, stopping as soon as
some well-defined conditions on the resulting reduction of the
objective function value are met. 
Since the evaluation of the objective function at the trial points 
can be performed in constant time (see Section~\ref{sec:nmarm}), line search methods are inexpensive 
in our context. Furthermore, from our numerical experience, using a non-monotone Armijo line search turned out to be the best
choice in practice. With this choice, a stepsize that yields a (safeguarded) growth of 
the objective function can be accepted (see e.g.~\citep{Gao2004,grippo1986nonmonotone,grippo1989truncated,GrippoSciandrone2002}).
 
The outline of our approach is given in Algorithm~\ref{NM-MFW}. At
each iteration~$k$, the algorithm first computes a descent direction,
choosing among a standard toward-step and an away-step direction, as
clarified in Section~\ref{sec:compdir}. Then, in case optimality
conditions are not satisfied, it calculates a stepsize along the given
direction by means of a non-monotone line search, see
Section~\ref{sec:nmarm}, updates the point, and starts a new
iteration.
\begin{algorithm}
  \caption{NM-MFW}
  \label{NM-MFW}
  \begin{algorithmic}
    \par\vspace*{0.1cm}
  \item$1$\hspace*{0.5truecm} Choose a suitable starting point $x^0\in S$
  \item$2$\hspace*{0.5truecm} {\bf For } $k=0,1,\ldots$
  \item$3$\hspace*{1.0truecm} Compute a descent direction $d^k$
  \item$4$\hspace*{1.0truecm} {\bf If} $\nabla f(x^k)^\top d^k = 0$ {\bf then} \textbf{STOP}
  \item$5$\hspace*{1.0truecm} Calculate a stepsize $\alpha^k \in(0,1]$ by means of a line search
  \item$6$\hspace*{1.0truecm} Set $x^{k+1}=x^k+\alpha^k d^k$
  \item$7$\hspace*{0.4truecm} {\bf End For }
    \par\vspace*{0.1cm}
  \end{algorithmic}
\end{algorithm}
    
In Section~\ref{sec:zero_opt}, we will discuss how to decide whether
the origin is an optimal solution of Problem~\eqref{ContRel0}. If this
is not the case, we always choose a starting point better than the
origin. The points $x^k$ produced at each iteration thus satisfy
$f(x^k)\le f(x^0)< f(0)$, so that $x^k \in \mathcal L(x^0)\cap S$ and
$0\not\in \mathcal L(x^0)\cap S$, where
$$\mathcal L(x^0) = \{x\in \R^n \mid f(x)\le f(x^0)\}.$$
This is done in order to avoid obtaining the origin in any of the following iterations, as the objective function may
not be differentiable in $x = 0$.

For the following, we summarize some important properties of Problem~\eqref{ContRel0}.
\begin{lemma}\label{basiclemma}
 Assume that $x=0$ is not an optimal solution of Problem~\eqref{ContRel0} and
 a point $x^0\in S$ exists such that $f(x^0)< f(0)$. Then,
\begin{itemize}
 \item[(a)] the set $\mathcal L(x^0)\cap S$ is compact;
 \item[(b)] the function $f$ is continuously differentiable in~$\mathcal L(x^0)\cap S$;
 \item[(c)] the function $h$ is Lipschitz continuous in $S$;
 \item[(d)] the function $f$ is Lipschitz continuous in $S$ with Lipschitz constant~$L \sqrt{\lambda_{max}(Q)} + \|\mu\|$, where 
$L$ is the Lipschitz constant of the function $h$.
\end{itemize}
\end{lemma}
\begin{proof}
  For~(a), it suffices to note that $\mathcal L(x^0)\cap S$ is a
  closed subset of the compact set $S$, while~(b) holds since
  $0\not\in \mathcal L(x^0)\cap S$. As $h$ is differentiable on the
  compact set $S$, we obtain~(c). Finally, to prove~(d), let~$Q^{1/2}$
  denote the unique symmetric matrix satisfying~$Q=Q^{1/2}Q^{1/2}$.
  Then
  \begin{equation}
    \begin{array}{l l}
      \|\nabla f(x)\| &= \Big\|h'(\|Q^{1/2} x\|) \frac{Qx}{\|Q^{1/2} x\|} - \mu\Big\|  \\
      \\
      &\le|h'(\|Q^{1/2} x\|)| \Big\| Q^{1/2} \frac{Q^{1/2}x}{\|Q^{1/2} x\|}\Big\| + \|\mu\| \\
      \\
      &\le|h'(\|Q^{1/2} x\|)| \| Q^{1/2}\| + \|\mu\| \\
      \\
      &\le L \sqrt{\lambda_{max}(Q)} + \|\mu\|.
    \end{array}
    \nonumber 
  \end{equation}
 
\end{proof}
In particular, it follows from~(d) that $f$ is uniformly continuous in $S$.

\subsection{Checking optimality in the origin}
\label{sec:zero_opt}

A first difficulty in dealing with Problem~\eqref{ContRel0} arises
from the fact that the objective function may not be differentiable in
the origin~$x=0$. We thus aim at checking, in a first phase of our
algorithm, whether the origin is an optimizer of
Problem~\eqref{ContRel0}. If so, we are done. Otherwise, our strategy
is to avoid the origin as an iterate of our algorithm, as discussed in
more detail in the following sections.

Since Problem~\eqref{ContRel0} is convex, the origin is a global
optimal solution if and only if there exists a
subgradient~$d\in \partial f(0)$ such that $d^\top x \ge 0$ for all
$x\in S$.  From standard results of convex analysis (see
e.g.~Theorem~2.3.9 in~\citet{clarke1990optimization}), we have that $\partial \|Q^\frac{1}{2} x\| = Q^{\frac 1 2 }$
and we derive that
$$
\partial f(0) = h'(0)\, Q^\frac{1}{2}B-\mu,
$$ 
where 
$B=\{w\in \R^n : \|w\|\le 1 \}$ is the unit ball in~$\R^n$.  Thus $x^\star
= 0$ is an optimal solution for Problem~\eqref{ContRel0} if and only
if
\begin{equation}\label{optCond0}
 \exists\; v\in B\colon\forall x\in S\colon \big(h'(0)\, Q^\frac{1}{2}v - \mu \big)^\top x \ge 0.
\end{equation}
Since $x\in S$ implies $x\ge 0$ and~$e_i\in S$ for all~$i=1,\dots, n$,
Condition~\eqref{optCond0} is equivalent to
\begin{equation}\label{optCond0bis}
 \exists\; v\in B\colon h'(0)\, Q^\frac{1}{2}v - \mu \ge 0.
\end{equation}
Note that Condition~\eqref{optCond0bis} is never satisfied
if~$h'(0)=0$, since $\mu\ge 0$ and $\mu \ne 0$. Consequently, the
origin is not an optimal solution of Problem~\eqref{ContRel0} in this
case. In general, Condition~\eqref{optCond0bis} allows to decide whether the origin is optimal 
by solving a convex quadratic optimization problem with non-negativity constraints,
namely
\[ 
  \begin{array}{l l}
\min & ||\tfrac 1{h'(0)}Q^{-\frac 12}(y+\mu)||\\[1ex]
\textnormal{~s.t. }& y\ge 0\;.
  \end{array}
\]

\subsection{Computation of a feasible descent direction}
\label{sec:compdir}

For the computation of a feasible descent direction we follow the
away-step approach described in~\citep{Guelat1982}.
At every iteration~$k$, we either choose a \emph{toward-step} or an
\emph{away-step}. We first solve the following
linearized problem (corresponding to the toward-step),
\begin{equation}\label{FWlin}
  \begin{array}{l l l}
    \hat x^k_{TS} = \arg & \min & \nabla f(x^k)^\top (x - x^k) \\
    &\textnormal{~s.t. }& x\in S,
  \end{array}
\end{equation}
and define $d^k_{TS}\in\R^n$ as $d^k_{TS} = \hat x^k_{TS} -
x^k$. 
The maximum stepsize that guarantees feasibility of the point chosen
along $d^k_{TS}$ is $\alpha_{TS} = 1$.
Once the toward-step direction is computed, we consider the problem corresponding to the away-step,
\begin{equation}\label{ASlin}
  \begin{array}{l l l}
    \hat x^k_{AS}  = \arg &\max & \nabla f(x^k)^\top (x - x^k) \\
    &\textnormal{~s.t. }& x\in S,\\
    & & x_i = 0 \; \mbox{ if } x^k_i = 0,
  \end{array}
\end{equation}
and define $d^k_{AS}\in\R^n$ as $d^k_{AS} = x^k - \hat x^k_{AS}$. 
In this case, the maximum stepsize guaranteeing feasibility is 
$$\alpha_{AS} = \max\{ \alpha\ge 0 \mid x^k + \alpha d^k_{AS} \in S \}.$$
If $\hat{x}_{AS}^k = e_{\hat \imath}$, the point $x^k + \alpha d^k_{AS}$ may become infeasible in case
the non-negativity constraint on $x_{\hat\imath}$ is violated.
On the other hand, if $\hat{x}_{AS}^k = 0$, the point $x^k + \alpha d^k_{AS}$ can only violate the constraint ${\bf 1}^\top x \le 1$.
Therefore, $\alpha_{AS}$ needs to be chosen as:
$$\alpha_{AS}:=\begin{cases}
  \begin{array}{ll}
    \frac{x_{\hat\imath}^k}{1 - x_{\hat\imath}^k}  & \textnormal{if }\hat x^k_{AS} = e_{\hat\imath},\\[1.3ex]
    \frac{1 - \mathbf{1}^\top x^k}{\mathbf{1}^\top x^k} & \textnormal{if }\hat x^k_{AS} = 0.
  \end{array}
\end{cases}$$
Note that, according to this rule, $\alpha_{AS} = 1$ may be an infeasible steplength.
Note also that, in case the equality constraints are not enforced in Problem~\eqref{ASlin}, 
$\alpha_{AS}$ could be trivially zero.

In order
to choose between the two directions, we use a criterion similar to the one presented
in~\citep{Guelat1982}: if
\begin{equation}\label{dirchoice}
  \nabla f(x^k)^\top d^k_{AS} \le \nabla f(x^k)^\top d^k_{TS} \quad\mbox{ and}\quad \alpha_{AS}>\beta, 
\end{equation}
with $0<\beta\ll1$ a suitably chosen constant value,
we choose the away-step direction, setting $\hat x^k = \hat
x^k_{AS}$ and $d^k =
x^k - \hat x^k = d^k_{AS}$. Otherwise we select
the toward-step direction, setting $\hat x^k = \hat x^k_{TS}$ and $d^k = \hat x^k - x^k = d^k_{TS}$.
The condition $\alpha_{AS}>\beta$ is needed to ensure convergence, as will become clear 
in Section~\ref{sec:conv_an} below. 

In both Problems~\eqref{FWlin} and~\eqref{ASlin}, we need to optimize
a linear function over a simplex. This reduces to computing the
objective function value at each vertex of the simplex, i.e.,
for~$0$ and~$e_1,\dots,e_n$ in~\eqref{FWlin} and for~$0$ and all~$e_i$
with~$x^k_i > 0$ in~\eqref{ASlin}. Consequently, after computing the
gradient~$\nabla f(x^k)$, both solutions can be obtained at a
computational cost of~$\mathcal{O}(n)$.

\subsection{Computation of a suitable stepsize}
\label{sec:nmarm}

When using exact line searches, the Frank-Wolfe method with away-steps
converges linearly if the objective function satisfies specific
assumptions; see e.g. \citep{Guelat1982, lacoste2013affine}. When an
exact line search approach is too expensive, we combine the away-step
approach with non-monotone inexact line searches. Even if the
Frank-Wolfe method is not guaranteed to converge linearly in the
latter case, it yields very good results in practice, as will be shown
in the numerical experience section.

In the non-monotone line search used in our algorithm, a stepsize is
accepted as soon as it yields a point which allows a sufficient
decrease with respect to a given reference value. A classical choice
for the reference value is the maximum among the last~$p_{nm}$ objective
function values computed, where~$p_{nm}$ is a positive integer constant.  See
Algorithm~\ref{NM-Arm} for the details of our line search method.
\begin{algorithm}
  \caption{Non-monotone Armijo line search}
  \label{NM-Arm}
  \begin{algorithmic}
    \par\vspace*{0.1cm}
  \item$0$\hspace*{0.5truecm} Choose $\delta \in (0,1)$, $\gamma_1\in (0,\frac 1 2)$, $\gamma_2\ge 0 $, $p_{nm}>0$. 
  \item$1$\hspace*{0.5truecm} Update $$\bar f^k = \max_{0\le i\le \min\{p_{nm},k\}} f(x^{k-i})$$
  \item$2$\hspace*{0.5truecm} Choose initial stepsize $\alpha\in(0,\alpha_{max}]$
  \item$3$\hspace*{0.5truecm} {\bf While} $f( x^k+ \alpha d^k )> \bar f^k+\gamma_1\, \alpha\, \nabla f(x^k)^{\top} d^k - \gamma_2\,\alpha^2\,\|d^k\|^2$
  \item$4$\hspace*{2.0truecm} Set $\alpha = \delta \alpha$
  \item$5$\hspace*{0.5truecm} {\bf End While}
    \par\vspace*{0.1cm}
  \end{algorithmic}
\end{algorithm}

The maximum stepsize $\alpha_{max}$ used in Line~$2$ of Algorithm~\ref{NM-Arm} is set
to $\alpha_{TS}$ if the toward-step direction is chosen; it is set to $\alpha_{AS}$, otherwise.

The following result states that Algorithm~\ref{NM-Arm} terminates in
a finite number of steps. It can be proved using similar arguments as
in the proof of Proposition~3 in~\citep{GrippoSciandrone2002}.
\begin{proposition}\label{prop1}
  For each $k$, assume that $\nabla f(x^k)^\top d^k < 0$.  Then
  Algorithm~\ref{NM-Arm} determines, in a finite number of iterations
  of the while loop in Lines 3--5, a stepsize $\alpha^k$ such that
  \begin{equation*}
    f( x^k+ \alpha^k d^k )\le \bar f^k+\gamma_1\, \alpha^k\, 
    \nabla f(x^k)^{\top} d^k - \gamma_2\,(\alpha^k)^2\,
    \|d^k\|^2.
  \end{equation*}
\end{proposition}

\noindent
From a practical point of view, it is important that the computation
of the objective function values of the trial points~$x^k+ \alpha d^k$
can be accelerated by using incremental updates. Therefore, during the
entire algorithm for solving Problem~\eqref{ContRel0}, we keep the
values~$Qx^k\in\R^n$, $(x^k)^\top Qx^k\in\R$, and $\mu^\top x^k\in\R$
up-to-date. In the line search, if a toward-step is applied and~$\hat x^k=e_{\hat\imath}$, we
exploit the fact that all expressions
\begin{eqnarray*}
  (x^k+ \alpha d^k)^\top  Q (x^k+ \alpha d^k) & = & (1-\alpha)^2\,(x^k)^\top  Q x^k  +  2 \alpha (1-\alpha)\, 
(Qx^k)_{\hat{\imath}} + \alpha^2  Q_{\hat{\imath} \hat{\imath}}\\
  \mu^\top (x^k+ \alpha d^k) & = &(1 - \alpha)\, \mu^\top x^k + \alpha\,  \mu_{\hat{\imath}}
\end{eqnarray*}
can be computed in constant time. Similarly, for~$\hat x^k=0$, we
obtain
\begin{eqnarray*}
  (x^k+ \alpha d^k)^\top  Q (x^k+ \alpha d^k) & = & (1-\alpha)^2\,(x^k)^\top  Q x^k\\
  \mu^\top (x^k+ \alpha d^k) & = &(1 - \alpha)\, \mu^\top x^k.
\end{eqnarray*}
In particular, if~$h$ can be evaluated in constant time, the same is
true for the computation of the objective value~$f(x^k+ \alpha
d^k)$. Moreover, when the line search is successful and the next
iterate is chosen, the same formula as above can be used to
compute~$(x^{k+1})^\top Qx^{k+1}\in\R$ and $\mu^\top x^{k+1}\in\R$ in
constant time, while~$Qx^{k+1}\in\R^n$ can be updated in linear time
using
\begin{equation*}
  Q(x^k+ \alpha d^k) = \begin{cases}
    \begin{array}{ll}
      (1 - \alpha)\,  Q x^k + \alpha\, Q_{\hat{\imath}\cdot} & \textnormal{ if }\hat x^k=e_{\hat\imath}\\
      (1 - \alpha)\,  Q x^k & \textnormal{ if }\hat x^k=0.
    \end{array}
  \end{cases}
\end{equation*}
The case of an away-step can be handled analogously.

In summary, after
computing $Qx^0\in\R^n$, $(x^0)^\top Qx^0\in\R$, and $\mu^\top
x^0\in\R$ from scratch, the computation of objective function values
takes~$\mathcal{O}(1)$ time per iteration of Algorithm~\ref{NM-Arm} --
assuming that~$h$ can be evaluated in constant time --
plus~$\mathcal{O}(n)$ time per iteration of Algorithm~\ref{NM-MFW} 
(needed to keep the values of~$Qx^k\in\R^n$, $(x^k)^\top Qx^k\in\R$, and $\mu^\top x^k\in\R$
up-to-date).

\subsection{Convergence analysis of the non-monotone Frank-Wolfe algorithm}\label{sec:conv_an}

We now analyze the convergence properties of the
non-monotone Frank-Wolfe algorithm~\texttt{NM-MFW} with away-steps
(Algorithm~\ref{NM-MFW}). 
All the proofs of the following theoretical results can be found in the Appendix.
\begin{lemma}\label{lemma1}
  Suppose that \texttt{NM-MFW} produces an infinite sequence $\{x^k\}_{k\in\N}$. Then
  \begin{itemize}
  \item[(i)] $x^k \in \mathcal L(x^0)\cap S$ for all $k$;
  \item[(ii)] the sequence $\{ \bar f^k\}_{k\in\N}$ is non-increasing and converges to a value $\bar f$.
 \end{itemize}
\end{lemma}
\begin{proof}
For the proof, see Appendix.
 
 \end{proof}

\begin{lemma}\label{lemma2}
 Suppose that \texttt{NM-MFW} produces an infinite sequence $\{x^k\}$. Then
\begin{equation*}
\lim_{k\rightarrow \infty} f(x^k) = \lim_{k\rightarrow \infty} \bar f^k = \bar f.
\end{equation*}
\end{lemma}
\begin{proof}
For the proof, see Appendix.
 
 \end{proof}

\begin{lemma}\label{prop:limgraddir}
  Suppose that \texttt{NM-MFW} produces an infinite sequence
  $\{x^k\}_{k\in\N}$. Then
  \begin{equation*}
    \lim_{k\rightarrow \infty} \nabla f(x^k)^\top d^k = 0.
  \end{equation*}
\end{lemma}
\begin{proof} 
For the proof, see Appendix.
 
\end{proof}

\begin{theorem}\label{Theo:convergence}
  Let $\{x^k\}\subseteq \mathcal L(x^0)\cap S$ be the sequence of
  points produced by \texttt{NM-MFW}. Then, either an integer $k\ge 0$
  exists such that $x ^k$ is an optimal solution for
  Problem~\eqref{ContRel0}, or the sequence~$\{x^k\}_{k\in\N}$ is
  infinite and every limit point $\xstar$ is an optimal solution for
  Problem~\eqref{ContRel0}.
\end{theorem}
\begin{proof} 
For the proof, see Appendix.
 
\end{proof}

We notice that, due to the use of the line search, there is no need to
make any particular assumption on the gradient of the objective
function (such as Lipschitz continuity) for proving
the convergence of Algorithm \texttt{NM-MFW}.

\subsection{Lower bound computation}
\label{sec:lbcomp}

When using Algorithm~\texttt{NM-MFW} within a branch-and-bound
framework as we will present in Section~\ref{bandb}, the availability of
valid dual bounds during the execution of~\texttt{NM-MFW} can help to
prune the current node before termination of the algorithm, and thus to
decrease the total running time of the branch-and-bound scheme.

Considering Problem~\eqref{ContRel0}, we can define the following dual
function~\citep{clarkson2010coresets,Jaggi2013} for all $x\in
S\setminus\{0\}$:
\begin{equation*}
  w(x):=\min_{z\in S} \big(f(x)+\nabla f(x)^\top (z - x)\big).
\end{equation*}
From the definition of~$w(x)$ and taking into account the convexity of
$f$, we have the following weak duality result:
\begin{equation}\label{duality}
  w(x)\leq f(x)+\nabla f(x)^\top (x^\star - x)\leq f(x^\star),
\end{equation}
where~$\xstar$ again denotes an optimal solution of
Problem~\eqref{ContRel0}. We thus obtain a dual bound in each
iteration for free, given by
\[
f(x^k)+\nabla f(x^k)^\top d^k\le w(x^k)=f(x^k)+ \min_{z\in S} \nabla f(x^k)^\top (z - x^k)=
f(x^k)+\nabla f(x^k)^\top d^k_{TS}.
\]
Note that this equation follows from how our direction is chosen, according to~\eqref{dirchoice} 
(see Section~\ref{sec:compdir} for further details). 
We can stop Algorithm~\texttt{NM-MFW}
as soon as~$f(x^k)+\nabla f(x^k)^\top d^k$ exceeds the current best upper bound in the
branch-and-bound scheme. Furthermore, strong duality holds
in~\eqref{duality} (in the sense that $w(x^\star) = f(x^\star)$); 
see e.g.~\citep{clarkson2010coresets} and the references therein.

\section{Branch-and-Bound algorithm}
\label{bandb}

In order to deal with integer variables in
Problem~\eqref{prob:gen_cap_bud}, we embedded Algorithm~\ref{NM-MFW}
into a branch-and-bound framework. Aiming at a fast enumeration of the
branch-and-bound tree, we follow the ideas that have been successfully
applied in, e.g.,~\citep{buchheim2016p}. In this
section, we give a short overview over the main features of the
branch-and-bound scheme.

\subsection{Branching and enumeration strategy}

At every node in our branch-and-bound scheme, we branch by fixing a
single integer variable to one of its feasible values. The enumeration
order of the children nodes is by increasing distance to the value of
this variables in the solution of the continuous relaxation $x^\star$,
computed by Algorithm~\ref{NM-MFW}. If the closest integer value to
$x^\star_i$ is $\lfloor x^\star_i\rfloor$, we thus consecutively fix
$x_i$ to integer values $\lfloor x^\star_i\rfloor,\lceil
x^\star_i\rceil, \lfloor x^\star_i\rfloor-1, \lceil
x^\star_i\rceil+1$, and so on. If the closest integer is $\lceil
x^\star_i\rceil$, we analogously start with fixing $x_i$ to 
the integer value
$\lceil x^\star_i\rceil$. By optimality of $x^\star$ and 
by the fact that the problem is convex,
the resulting lower bounds are non-decreasing when fixing to either
increasing values greater than $x^\star_i$ or decreasing values less
than $x^\star_i$. In particular, when being able to prune a node, all
siblings beyond this node can be pruned as well.

Once we arrive at level~$|I|$, all
integer variables are fixed and the problem reduces to the purely
continuous problem~\eqref{ContRel}. We refer to~\cite{{buchheim2016p}} 
and the references therein for further details on the branching strategy. 
\subsection{Lower bounds after fixing}

An advantage of branching by fixing variables as opposed to branching
by splitting up variable domains is that the subproblems in the
enumeration process of the search tree essentially maintain the same
structure.  Fixing a variable in Problem~\eqref{ContRel} just
corresponds to moving certain coefficients from the matrix~$M$ to a
linear or constant part under the square root, and from the vector~$r$
to a constant part outside the square root. More precisely, assume
that the variables with indices in~$J\subseteq I$ have
been fixed to values~$s=(s_i)_{i\in I}$. The problem then reduces to the
minimization of
\begin{equation}\label{eq:probfix}
  f_s:\mathbb{Z}^{|I|-|J|}\times\mathbb{R}^{n-|I|}\rightarrow\mathbb{R},\ x\mapsto h\Big(\sqrt{x^{\top}M_s x+ c_s^{\top}x+ d_s}\Big)-r_s^{\top}x-t_s
\end{equation}
over the feasible region
$\mathcal{F}_s=\{x\in\mathbb{Z}^{|I|-|J|}\times\mathbb{R}^{n-|I|}\mid
a_s^{\top}x\leq b_s,\ x\geq 0\}$, where the matrix $M_s$ is obtained
by deleting the rows and columns corresponding to~$J$, the
vector $a_s$ is obtained by deleting the columns corresponding to~$J$,
and the remaining terms are updated appropriately.

Note that the relaxation of Problem~\eqref{eq:probfix} has a slightly more general form
than the original Problem~\eqref{ContRel}, since the data~$c_s$
and~$d_s$ may be non-zero as a result of fixing variables. However,
the algorithm for solving Problem~\eqref{ContRel} discussed in
Section~\ref{sec:FW} can easily be applied to the relaxation of
Problem~\eqref{eq:probfix} as well, the only difference being in the
computation of the gradient. In fact, in case at least one variable
has been fixed to a non-zero value, we obtain~$d_s>0$ since~$M\succ
0$. In particular, the objective function becomes globally
differentiable in this case.

\subsection{Upper bounds}

As an initial upper bound in the branching tree, we use a simple
heuristic, adapted from a greedy heuristic by~\citet{julstrom2005} for
the quadratic knapsack problem.  Analogously to the notation used in
the theory of knapsack problems the profit ratio $p_i$ of an item $i$
is defined as the sum of all profits that one gains by putting item
$i$ into the knapsack, divided by its weight. Transferred to our
application, we have
$$p_i:=\Bigg(h\Big(\sqrt{m_{ii}+\textstyle 2\sum_{j\neq i}m_{ij}}\Big)-r_i\Bigg)/a_i$$
for all $i=1,\dots,n$. Julstrom proposed to sort all items in a non-decreasing 
order with respect to~$p_i$ and, starting from the first item, successively set $x_i=1$ until the capacity 
of the knapsack is reached. The remaining variables are set to zero.

We adapt this algorithm by allowing multiple copies of each item, i.e. $x_i=\lfloor\frac{\bar b}{a_i}\rfloor$, 
where $\bar b$ is the current capacity of the knapsack. 

During the branch-and-bound enumeration, we do not use any heuristics
for improving the primal bound, since the fast enumeration using a
depth-first search usually leads to the early identification of good
feasible solutions and hence to fast updates of the bound.  Once all
integer variables have been fixed, we compute the optimal solution of
the subproblem in the reduced continuous subspace.
 
\subsection{Warmstarts}

With the aim of speeding-up our branch-and-bound scheme, we use a
warmstart procedure by taking over information from the parent node.
For this, let $x^\star\in\R^d$ be the optimal solution in the parent node and
define $\tilde x\in\R^{d-1}$ by removing the entry of~$x^\star$ corresponding to
the variable that has been fixed last. If $\tilde
x$ is feasible for the current node relaxation, we always use it as a
starting point for \texttt{NM-MFW}, otherwise we choose one of the
following feasible points according to our chosen warmstarting rule:
\begin{itemize}
\item the first unit vector $e_1 = (1,0,\dots,0)\in\R^{d-1}$;\\[-1ex]
\item the projection $\tilde x_p$ of $\tilde x$ onto the feasible region;\\[-1ex]
\item or the unit vector
  $e_{\hat\imath}$ with
  ${\hat\imath}:=\textnormal{argmin}_i h\Big(\sqrt{m_{ii} + \sum_{j\neq i}
    2m_{ij}}\Big)-r_i$.
\end{itemize}
The resulting warmstarting rules are denoted by ($\tilde x \vee e_1$), ($\tilde x \vee \tilde x_{p}$), and ($\tilde x \vee e_{\hat\imath}$), respectively. This notation is meant to emphasize that we either use $\tilde x$ or, if not possible, one of the other choices depending on the selected rule.

Note that the point $\tilde x_p$ can be computed by the algorithm originally
proposed by~\citet{held1974} that was recently rediscovered
by~\citet{duchi2008}.  For the latter version the overall complexity
has been proved to be $\mathcal{O}(n^2)$. The unit vector~$e_{\hat\imath}$ is chosen by again adapting ideas of the greedy heuristic
by~\citet{julstrom2005}. It represents the vertex of $S$ where the 
potential increase of the objective function due to the remaining items 
$j\neq \hat\imath$ is minimized, if setting $x_{\hat\imath}=1$.

\section{Numerical experience}\label{sec:exp}

In order to investigate the potential of our algorithm \texttt{FW-BB} when
applied to Problem~\eqref{prob:gen_cap_bud}, we
implemented it in~C++ and Fortran~90 and performed an extensive
computational evaluation. As benchmark data set, we used historical
real-data capital market indices from the Standard\,\&\,Poor's~500
index (S\&P~500) that were used and made public
by~\citet{cesarone2009}. This data set was used for solving a
\emph{Limited Asset Markowitz (LAM)} model. For each of the 500 stocks
the authors obtained 265 weakly price data, adjusted for dividends,
from Yahoo Finance for the period from March 2003 to March 2008.
Stocks with more than two consecutive missing values were
disregarded. The missing values of the remaining stocks were
interpolated, resulting in an overall of 476 stocks. Logarithmic
weekly returns, expected returns and covariance matrices were computed
based on the period March 2003 to March 2007.

By choosing stocks at random from the 476 available ones, we built 
mixed-integer portfolio optimization instances of different sizes.  
Namely, we built
10 problems with 100, 10 with 150 and 10 with 200 stocks, considering
$|I|=\lfloor n/2\rfloor$ (so half of the variables are constrained to be integer). 
 We considered three different values
for~$b$, representing the budget of the investor, namely
$b_1:=1\cdot\sum_{i=1}^na_i$, $b_2:=10\cdot\sum_{i=1}^na_i$, and
$b_3:=100\cdot\sum_{i=1}^na_i$, yielding a total of 90 instances.

All experiments were carried out on Intel Xeon processors running at
2.60 GHz. All running times were measured in cpu seconds and the
time-limit was set to one cpu hour. In the following, we first present
a numerical evaluation related to our algorithm \texttt{FW-BB}: we
explore the benefits obtained from using the non-monotone line search
and using warmstart alternatives. Then, we present a comparison of
\texttt{FW-BB} with the MISOCP and the MIQP solver of \texttt{CPLEX 12.6},
for the two cases~$h(t)=\Omega t$ and $h(t)=t^2$, respectively.
Finally, to show the generality of our approach, we report the results
of numerical tests for a non-standard risk-weighting function~$h$.

\subsection{Benefits of the non-monotone line search and warmstarts}

The \texttt{NM-MFW}-algorithm devised in Section~\ref{sec:FW}
uses a non-monotone line search; in our implementation of
\texttt{FW-BB} we set $p_{nm}=1$. In order to show the benefits of the
non-monotone version of \texttt{FW-BB} we report in Table~\ref{tab:nm}
a comparison between the non-monotone version (\texttt{NM-FW-BB}) and
the monotone one (\texttt{M-FW-BB}), on instances with
$h(t)=\Omega t$ and budget constraint~$a^\top x\le b_3$. 
We considered $(\tilde x \vee \tilde x_p)$ as warmstart choice.
In Table~\ref{tab:nm} we report, for each dimension, the number of
instances solved within the time limit ($\#$), the average running
times (time), and the average numbers of iterations of \texttt{NM-MFW} 
in each node of the enumeration tree (it). 
All averages are taken over the set of instances solved within the 
time limit. 
Using the non-monotone line search,
\texttt{FW-BB} is able to solve a greater number of instances within the time limit.
Furthermore, \texttt{NM-FW-BB} gives in general better performance in terms 
of running times, while the number of iterations is very similar, showing the advantage 
of allowing stepsizes with a safeguarded growth of the objective function.

\begin{table}
	\centering      
	\begin{tabular}{|c||r|r|r||r|r|r|} 
	  \hline
	  \multicolumn{1}{|c||}{} & \multicolumn{3}{c||}{\texttt{NM-FW-BB}} & \multicolumn{3}{c|}{\texttt{M-FW-BB}}\\
	  $n$ & \# & time & it & \# & time & it\\
	  \hline
	  \hline
	  100  & 10  &   1.6  & 314.3 & 10 &  0.8    & 294.9\\
	  150  & 10  &   7.1  & 307.9 & 9  &  69.0   & 300.7\\
	  200  & 8   &   32.4 & 277.8 & 8  &  340.7  & 256.0\\
	  \hline
	\end{tabular}\vspace*{0.1cm}
	\caption{Comparison between non-monotone and monotone version of \texttt{FW-BB} on instances with $h(t)=\Omega t$, $\varepsilon=0.95$, $b=b_3$.}
	\label{tab:nm}  
    \end{table}

In order to investigate the benefits of the warmstart choices $(\tilde
x \vee e_1)$, $(\tilde x \vee \tilde x_p)$, $(\tilde x \vee
e_{\hat\imath})$, we again ran the different versions of \texttt{FW-BB} on instances with $h(t)=\Omega t$ and
budget constraint~$a^\top x\le b_3$.  We compare the three warmstart
possibilities presented above with the following alternatives:
\begin{tabbing}
($e_1$)~~ \= always choose $e_1$;\\
($e_{\hat\imath}$) \> always choose $e_{\hat\imath}$.
\end{tabbing}
In Table~\ref{tab:mixed-warmstart} we show the results related to the
five different starting point choices.
We can observe that the best choice  among those considered, according to the number of instances 
solved within the time limit, is~$(\tilde x \vee
\tilde x_{p})$.
We also observe that, when $n=200$,
choosing $(\tilde x \vee e_{1})$ is better than considering $e_{1}$ or
$e_{\hat\imath}$ as starting points, highlighting the benefits of
using warmstarts.

\begin{table}
      \centering
      \begin{tabular}{|c||r|r||r|r||r|r||r|r||r|r|}
        \hline
        \multicolumn{1}{|c||}{} & \multicolumn{2}{c||}{\texttt{$e_1$}} & \multicolumn{2}{c||}{\texttt{$e_{\hat\imath}$}} 
        & \multicolumn{2}{c||}{\texttt{$\tilde x\vee e_1$}} & \multicolumn{2}{c||}{\texttt{$\tilde x\vee e_{\hat\imath}$}}
        & \multicolumn{2}{c|}{\texttt{$\tilde x\vee \tilde x_p$}} \\
        $n$ & \# & time  & \# & time  &\# & time  & \# & time &  \# & time  \\
        \hline
        \hline
        100 & 10 & 0.2  & 10 & 0.8   & 10 & 0.8  & 10& 1.6  & 10 & 0.5  \\
        150 & 10 & 3.8  & 10 & 3.9   & 10 & 7.1  & 10& 7.1  & 10 & 6.1  \\
        200 & 7 & 220.2 & 7  & 223.5 & 8  & 31.7 & 8 & 32.4 & 9 & 46.0 \\
        \hline
      \end{tabular}\vspace*{0.1cm}
      \caption{Comparison on different warmstart strategies on instances with $h(t)=\Omega t$, $\varepsilon=0.95$, $b=b_3$.}
      \label{tab:mixed-warmstart}
  \end{table}

\subsection{Comparison with \texttt{CPLEX 12.6}}

In this section, we present a numerical comparison on instances with
$h(t)=\Omega t$ and $h(t)=t^2$. We compare \texttt{FW-BB} with the
MISOCP and the MIQP solver of \texttt{CPLEX 12.6},
respectively. Concerning \texttt{FW-BB}, we consider the two
non-monotone versions, \texttt{FW-BB-P} and \texttt{FW-BB-G}, using
$(\tilde x \vee \tilde x_{p})$ and $(\tilde x \vee e_{\hat\imath})$,
respectively.  We use an absolute optimality tolerance of $10^{-10}$
for all algorithms.

\paragraph{Comparison on instances with $h(t)= \Omega t$.}
In order to compare \texttt{FW-BB} with \texttt{CPLEX 12.6}, 
we modeled~\eqref{prob:gen_cap_bud} as an equivalent mixed-integer second-order cone program (MISOCP):
  \begin{align*}
  -\min\left\{y-r^{\top}x:a^\top x \leq b,\ \Omega\sqrt{x^{\top}M x}\leq y,\ x\geq 0,\ x_i\in \Z,\ i=1,\dots,|I|,\ y\in\R\right\}.
  \end{align*}
We chose $\Omega=\sqrt{(1-\varepsilon)/\varepsilon}$, where
$\varepsilon\in\{0.91,0.95,0.99\}$.  The value of $\varepsilon$
controls the amount of risk the investor is willing to take.  In
theory, $\varepsilon$ can take any value in (0,1], where a small value
  implies a big weight on the risk-term and $\varepsilon=1$ means
  that the risk is not taken into account. Numerical tests on single
  instances showed that any value of $\varepsilon$ in (0,0.9] leads to
    the trivial optimal solution zero, i.e.\ not investing anything is
    the optimal decision for the investor.  Therefore, we restricted
    our experiments to the three values of $\varepsilon$ mentioned
    above.

In~Table~\ref{tab:mixOm}, we report for each algorithm the following
data: numbers of instances solved within the time limit ($\sharp$), average
running times (time), average numbers of branch-and-bound nodes (nodes). All
averages are taken over the set of instances solved within the time
limit. 
We show the computational results for the
three different values of $\varepsilon$ and $b$. We
can see that \texttt{FW-BB} suffers from an increasing right
hand side~$b$,  which however holds for~\texttt{CPLEX 12.6} as well,
even to a larger extent. The choice of~$\varepsilon$ does not
significantly effect the performance of~\texttt{FW-BB}, while~\texttt{CPLEX 12.6} 
performs better on instances
with large~$\varepsilon$.  Altogether, we can observe that~\texttt{FW-BB-P} 
is able to solve the largest number of instances
within the time limit.  When the number of solved instances is the same,
both version of~\texttt{FW-BB} outperform the MISOCP solver of~\texttt{CPLEX 12.6} in terms of cpu time.
Note that the average number of branch-and-bound nodes in~\texttt{FW-BB} is much larger than 
that needed by~\texttt{CPLEX 12.6}. This highlights how solving the continuous relaxations 
by~\texttt{NM-FW-BB} leads to a fast enumeration of the branch-and-bound nodes.
Besides~Table~\ref{tab:mixOm}, we  visualize our running time results by performance profiles  
in  Figure~\ref{fig:perfprof}, as proposed in~\cite{PP2002}. They  confirm that, in terms of 
cpu time,~\texttt{FW-BB-P} outperforms the MISOCP solver of~\texttt{CPLEX 12.6} significantly.
 
In our experiments, we noticed that in some cases \texttt{FW-BB} and
\texttt{CPLEX} provide slightly different minimizers, yielding
slightly different optimal objective function values. While on certain
instances the optimal solution of \texttt{FW-BB} is slightly superior
to \texttt{CPLEX}, on other instances it is the other way round. We
observed a relative difference from the best solution of the order of 
$10^{-3}$.

    \begin{table}
	\centering      
	\begin{tabular}{|c|c|c||r|r|r||r|r|r||r|r|r|} 
	  \hline
	  \multicolumn{3}{|c||}{inst} & \multicolumn{3}{c||}{\texttt{FW-BB-P}} & \multicolumn{3}{c||}{\texttt{FW-BB-G}}
          & \multicolumn{3}{c|}{\texttt{\texttt{CPLEX 12.6}}} \\
	  $n$ & $\varepsilon$ & $b$ & \# & time & nodes & \# & time & nodes & \# & time & nodes \\
	  \hline
	  \hline
	  100 &  0.91 & $b_1$ & 10 & 0.17 & 1.61e+03 & 10 & 0.33 & 1.63e+03 & 10 & 17.00  & 3.81e+03\\
	  100 &  0.91 & $b_2$ & 10 & 0.09 & 8.29e+02 & 10 & 0.22 & 8.33e+02 & 10 & 279.15 & 7.90e+03\\
	  100 &  0.91 & $b_3$ & 10 & 0.30 & 3.74e+02 & 10 & 0.42 & 4.28e+02 & 3 & 51.01  & 2.77e+03 \\ \hline
	  100 &  0.95 & $b_1$ & 10 & 0.02 & 2.59e+02 & 10 & 0.04 & 2.65e+02 & 10 & 1.89   & 4.66e+02  \\
	  100 &  0.95 & $b_2$ & 10 & 0.04 & 3.19e+02 & 10 & 0.09 & 3.14e+02 & 10 & 59.10  & 2.98e+03 \\
	  100 &  0.95 & $b_3$ & 10 & 0.47 & 5.87e+02 & 10 & 1.57 & 2.08e+03 & 5 & 364.90   & 4.39e+03 \\ \hline
	  100 &  0.99 & $b_1$ & 10 & 0.01  & 1.70e+02 & 10 & 0.01 & 1.78e+02 & 10 & 0.15    & 3.70e+01 \\
	  100 &  0.99 & $b_2$ & 10 & 0.04  & 5.81e+02 & 10 & 0.04 & 6.64e+02 & 10 & 0.70    & 1.85e+02  \\
	  100 &  0.99 & $b_3$ & 10 & 16.51 & 1.57e+04 & 10 & 260.60 & 3.71e+05& 9 & 503.62     & 1.03e+04\\ \hline\hline
	  150 &  0.91 & $b_1$ & 10 & 0.14  & 6.56e+02 & 10  & 4.45 & 1.06e+04   & 10 & 52.53     & 3.18e+03\\
	  150 &  0.91 & $b_2$ & 10 & 0.40 & 1.73e+03  & 10  & 46.4 & 4.83e+04   & 6 & 707.94     & 6.70e+03 \\
	  150 &  0.91 & $b_3$ & 10 & 2.15 & 2.01e+03  & 9   & 1.77 & 1.41e+03   & 5 & 47.32  & 1.81e+03  \\ \hline
	  150 &  0.95 & $b_1$ & 10 & 0.15 & 9.76e+02 & 10 & 0.24 & 1.05e+03 & 10 & 11.49    & 1.04e+03\\
	  150 &  0.95 & $b_2$ & 10 & 0.17 & 6.75e+02 & 10 & 5.70 & 8.89e+03 & 8 & 225.78   & 3.04e+03\\
	  150 &  0.95 & $b_3$ & 10 & 6.14 & 6.15e+03 & 10 & 7.11 & 6.16e+03 & 5 & 834.79   & 6.23e+03\\ \hline
	  150 &  0.99 & $b_1$ & 10 & 0.04 & 2.56e+02 & 10 & 0.06 & 2.66e+02 & 10 & 35.08    & 5.81e+02 \\
	  150 &  0.99 & $b_2$ & 10 & 0.10 & 2.20e+02 & 10 & 0.23 & 5.08e+02  & 10 & 6.69    & 6.14e+02 \\
	  150 &  0.99 & $b_3$ & 10 & 0.78 & 8.67e+02 & 10 & 0.82 & 8.80e+02 & 9 & 422.15    & 3.53e+03\\ \hline\hline
	  200 &  0.91 & $b_1$ & 10 & 4.81 & 1.71e+04  & 10 & 5.80   & 1.35e+04& 10 & 465.62    & 9.48e+03 \\
	  200 &  0.91 & $b_2$ & 9  & 19.83 & 7.89e+04 & 9  & 116.78 & 1.89e+05 & 3 & 879.46   & 9.14e+03\\
	  200 &  0.91 & $b_3$ & 10 & 22.99 & 1.86e+04 & 10 & 32.44  & 2.20e+04 & 3 & 204.92   & 3.80e+03 \\ \hline
	  200 &  0.95 & $b_1$ & 10 & 0.37 & 1.33e+03 & 10 & 0.54 & 1.29e+03 & 10 & 75.50    & 3.46e+03 \\
	  200 &  0.95 & $b_2$ & 10 & 0.82 & 1.38e+03 & 10 & 1.40 & 1.48e+03 & 5 & 44.64 & 1.55e+03  \\
	  200 &  0.95 & $b_3$ & 9 & 45.98 & 3.74e+04 & 8 & 32.39 & 2.05e+04 & 7 & 38.77 & 2.42e+03 \\ \hline
	  200 &  0.99 & $b_1$ & 10 & 2.17  & 1.57e+04 & 10 & 2.00 & 1.57e+04 & 10 & 2.44    & 2.76e+03 \\
	  200 &  0.99 & $b_2$ & 10 & 0.49  & 6.04e+02 & 10 & 0.95 & 9.12e+02 & 9 & 277.08  & 2.61e+03 \\
	  200 &  0.99 & $b_3$ & 10 & 11.14 & 1.06e+04 & 9  & 67.57 & 5.90e+04 & 9 & 183.80 & 2.13e+03\\ 
	  \hline
	\end{tabular}\vspace*{0.1cm}
	\caption{Comparison of \texttt{FW-BB} and \texttt{CPLEX 12.6} on instances with $h(t)=\Omega t$.}
	\label{tab:mixOm}  
    \end{table}

\begin{figure}[h!]
  \begin{center}
    \psfrag{eps = 0.91}[cc][cc]{$\varepsilon=0.91$}
    \psfrag{eps = 0.95}{$\varepsilon=0.95$}
    \psfrag{eps = 0.99}{$\varepsilon=0.99$}
    \psfrag{CPLEX 12.6}[lc][lc]{\scriptsize\texttt{CPLEX 12.6}}
    \psfrag{FW-BB-P}[lc][lc]{\scriptsize\texttt{FW-BB-P}}
    \psfrag{FW-BB-G}[lc][lc]{\scriptsize\texttt{FW-BB-G}}
    \includegraphics[trim = 1.5cm 0cm 1.5cm 0mm, clip, width=1.0\textwidth]{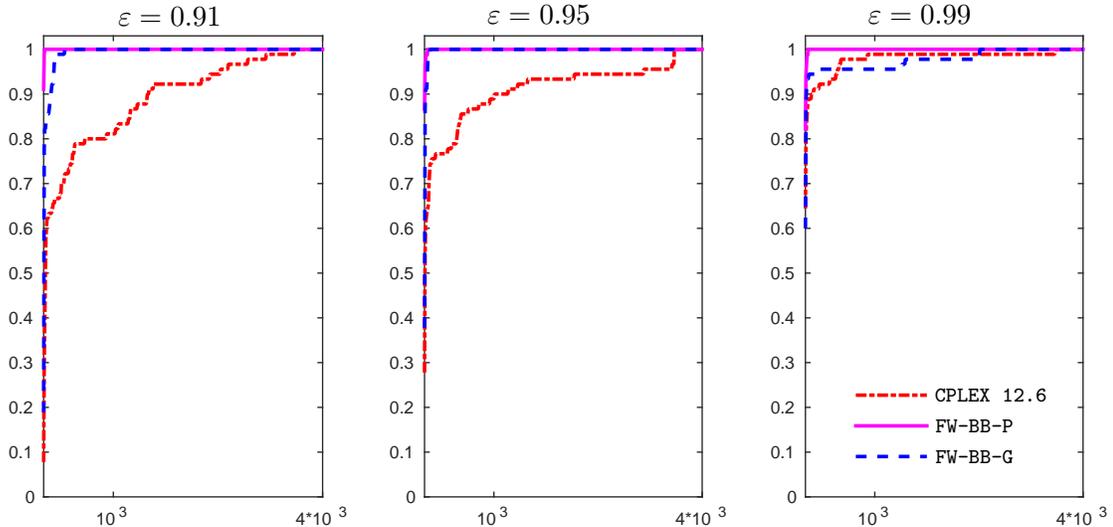}\\
    \caption{Comparison of \texttt{FW-BB} and \texttt{CPLEX 12.6}: performance profiles with respect to running times for different values of $\varepsilon$.}
    \label{fig:perfprof}
  \end{center}
\end{figure}

\paragraph{Comparison on instances with $h(t)= \Omega t^2$.}
If we consider as risk-weighting function $h(t)= \Omega t^2$, Problem~\eqref{prob:gen_cap_bud} 
reduces to a convex quadratic mixed-integer problem, and the objective function is differentiable
everywhere in the feasible set.
In Table~\ref{tab:mixQuad} we report the comparison among \texttt{FW-BB-P}, \texttt{FW-BB-G} and the MIQP solver of \texttt{CPLEX 12.6}.
We considered $\Omega = 1$.
All algorithms were able to solve all the instances very quickly.
The MIQP solver of \texttt{CPLEX 12.6} shows the best cpu times,
although both versions of \texttt{FW-BB}
are also very fast, even if they enumerate a higher number of
nodes. 

We would like to remark that our branch-and-bound algorithm does not
exploit the (quadratic) structure of the objective function, since it
is  designed to solve a more general class of problems than MIQPs.
Nevertheless, the algorithm gives competitive results also when dealing with those problems.

	\begin{table}
      \centering    
	\begin{tabular}{|c|c||r|r|r||r|r|r||r|r|r|} 
	  \hline
	  \multicolumn{2}{|c||}{inst} & \multicolumn{3}{c||}{\texttt{FW-BB-P}} & \multicolumn{3}{c||}{\texttt{FW-BB-G}} & \multicolumn{3}{c|}{\texttt{CPLEX 12.6}} \\
	  $n$ & $b$ & \# & time & nodes & \# & time & nodes & \# & time & nodes \\
	  \hline
	  \hline
	  100 & $b_1$ & 10 &     0.06 & 4.10e+02 & 10 &     0.06 & 4.10e+02 & 10 & 0.04 & 1.12e+01 \\
	  100 & $b_2$ & 10 &     0.07 & 6.38e+02 & 10 &     0.07 & 6.38e+02 & 10 & 0.03 & 1.58e+01 \\
	  100 & $b_3$ & 10 &     0.23 & 9.86e+02 & 10 &     0.23 & 9.86e+02 & 10 & 0.03 & 2.59e+01 \\ \hline
	  150 & $b_1$ & 10 &     0.12 & 7.66e+02 & 10 &     0.12 & 7.65e+02 & 10 & 0.06 & 1.92e+01\\
	  150 & $b_2$ & 10 &     0.19 & 9.76e+02 & 10 &     0.18 & 9.76e+02 & 10 & 0.06 & 2.06e+01\\
	  150 & $b_3$ & 10 &     0.19 & 8.80e+02 & 10 &     0.19 & 8.81e+02 & 10 & 0.06 & 9.90e+00\\ \hline
	  200 & $b_1$ & 10 &     0.61 & 3.28e+03 & 10 &     0.61 & 3.28e+03 & 10 & 0.11 & 2.07e+01\\
	  200 & $b_2$ & 10 &     0.94 & 5.24e+03 & 10 &     0.91 & 5.24e+03 & 10 & 0.11 & 1.91e+01 \\
	  200 & $b_3$ & 10 &     0.41 & 1.46e+03 & 10 &     0.42 & 1.46e+03 & 10 & 0.12 & 2.40e+01\\ 
	  \hline
	\end{tabular}\vspace*{0.1cm}
	\caption{Comparison of \texttt{FW-BB} and \texttt{CPLEX 12.6} on instances with $h(t)=t^2$.}
	\label{tab:mixQuad}
      \end{table}

\subsection{Results with a non-standard risk-weighting function}

As a further experiment, we tested our instances
considering a different risk-weighting function $h:\R_+ \rightarrow
\R$, namely
\[h_{exp}(t)=\begin{cases}0 &\ t\leq\gamma \\ \exp(t-\gamma)-(t-\gamma+1) &\ t>\gamma,\end{cases}\]
such that the investor's risk-aversion increases exponentially in the
risk after exceeding a certain threshold value~$\gamma$. In
Table~\ref{tab:minth}, we report the results of \texttt{FW-BB-P},
considering three choices~$\gamma \in\{ 0,\, 1,\, 10\}$.  We observe
that for both~$\gamma = 0$ and $\gamma = 1$ our algorithm
\texttt{FW-BB-P} is able to solve all instances within the time limit,
and that instances get more difficult for \texttt{FW-BB-P} with
increasing~$\gamma$.

      \begin{table}[h!]
	\centering  
	\begin{tabular}{|c|c||r|r|r||r|r|r||r|r|r|} 
	  \hline
	  \multicolumn{2}{|c||}{inst} & \multicolumn{3}{c||}{\texttt{$\gamma = 0$}} & 
	   \multicolumn{3}{c||}{\texttt{$\gamma = 1$}} 
	  & \multicolumn{3}{c|}{\texttt{$\gamma = 10$}}\\
	  $n$ & $b$ & \# & time & nodes & \# & time & nodes & \# & time & nodes  \\
	  \hline
	  \hline

  100	&	$b_1$	&	10	&	0.09	&	4.8e+02	&	10	&	0.17	&	7.2e+02	&	10	&	0.09	&	5.4e+02		\\
  100	&	$b_2$	&	10	&	0.07	&	4.1e+02	&       10	&	0.27	&	7.5e+02	&	10	&	243.87	&	3.3e+05		\\
  100	&	$b_3$	&	10	&	0.30	&	8.6e+02	&	10	&	31.14	&	5.1e+04	&	5	&	401.57	&	6.2e+05		\\ \hline
  150	&	$b_1$	&	10	&	0.17	&	9.0e+02	&	10	&	1.31	&	4.5e+03	&	10	&	0.19	&	3.2e+02		\\
  150	&	$b_2$	&	10	&	0.33	&	1.5e+03	&	10	&	2.52	&	7.9e+03	&	10	&	193.40	&	2.0e+04		\\
  150	&	$b_3$	&	10	&	0.56	&	2.2e+03	&	10	&	6.50	&	1.2e+04	&	4	&	565.76	&	7.1e+05		\\ \hline
  200	&	$b_1$	&	10	&	1.41	&	6.6e+03	&	10	&	14.96	&	4.7e+04	&	10	&	7.40	&	7.7e+03		\\
  200	&	$b_2$	&	10	&	1.09	&	3.2e+03	&	10	&	50.47	&	1.1e+05	&	7	&	929.46	&	7.9e+05		\\
  200	&	$b_3$	&	10	&	0.82	&	2.5e+03	&	10	&	30.25	&	3.0e+04	&	5	&	138.07	&	1.5e+05		\\

	  \hline
	\end{tabular}\vspace*{0.1cm}
	\caption{Results with an exponential risk-weighting function.}
	\label{tab:minth}  
    \end{table}

In order to investigate the influence of the risk-weighting function on the
optimal solution, we compared different functions for an
instance of dimension $n = 100$ under the constraint $a^\top x \le
b_1$. The results are given in Table~\ref{tab:ofvh}. We report, for
each risk-weighting function $h(t)$ depending on a specific risk
parameter (risk-par), the objective function value obtained (obj), the
value of the return term in the objective function evaluated at the optimal solution ($r^\top x^\star$),
the number of non-zero entries in the optimal
solution~($\|x^\star\|_0$), and the maximal entry in the optimal
solution~($\|x^\star\|_\infty$).
\begin{table} [h!]
  \centering     
  \begin{tabular}{|c|l||r|r|r|r|} 
    \hline
    $h(t)$ & risk-par & obj & $r^\top x^\star$ & $\|x^\star\|_0$ & $\|x^\star\|_\infty$\\
    \hline\hline
    $\Omega\, t$ & $\eps = 0.91 $ & 0.3684 & 2.2452 & 16 & 58 \\
    & $\eps = 0.95 $ & 1.4454& 6.0911 & 4 & 280\\
    & $\eps = 0.99 $ & 4.2161& 6.4523 & 3 & 320\\
    \hline
    $\Omega\, t^2$  & $ \Omega = 1$ &0.0513 & 0.1021& 16 & 2\\
    \hline
    $h_{exp}$ &$\gamma = 0$ & 0.0905 & 0.1715& 15 & 3.14 \\
    &$\gamma = 1$ & 0.5258 & 0.5900& 16 & 11.87\\
    &$\gamma = 10$ & 3.4991 & 3.5348& 7 & 113\\
    \hline
  \end{tabular}\vspace*{0.1cm}
  \caption{Results on a mixed-integer instance with $n=100$ for different risk-weighting functions.}
  \label{tab:ofvh}  
\end{table}

Not surprisingly, the results show that a larger weight on the risk-term
leads to a smaller expected return in the optimal solution. At the same time, a
large weight on the risk favors a diversified portfolio, so that the
number of non-zeros increases with the weight on the risk, at the same and 
the maximal amount invested into a single investment
decreases. However, the precise dependencies are defined by the
function~$h$. In Figure~\ref{fig:contour}, we show contour plots
for the different types of functions~$h(t)$ considered here.
\begin{figure}[h!]
  \begin{center}
     \includegraphics[trim = 2.9cm 1.0cm 2.0cm 0mm, clip, width=1.0\textwidth]{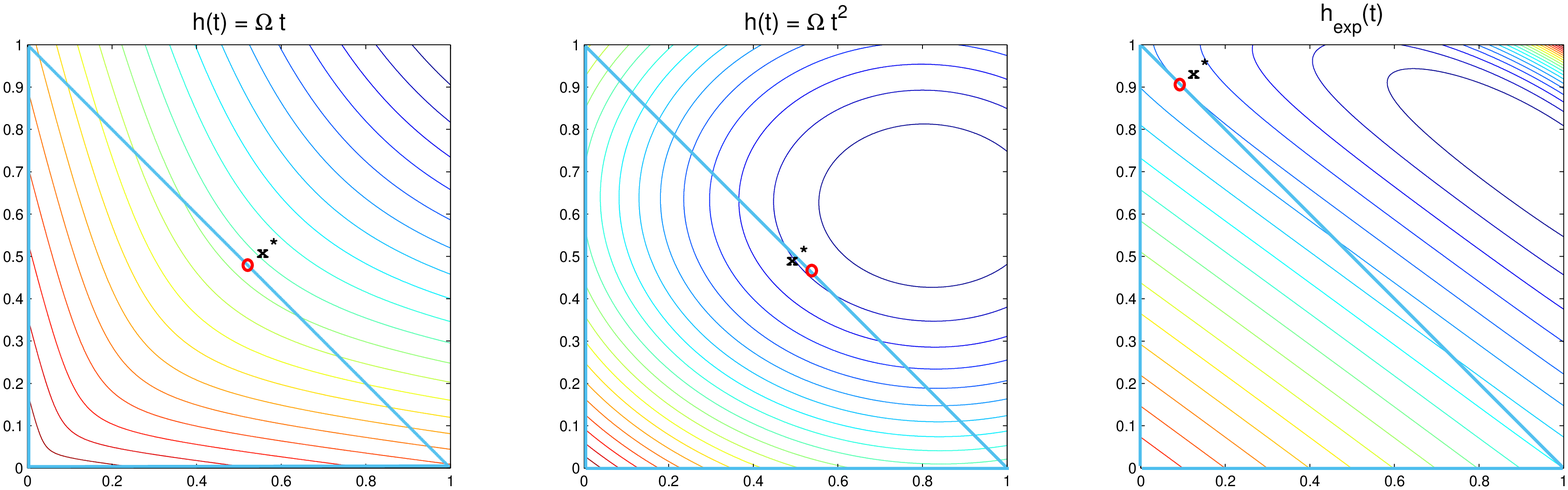}\\
    \caption{Contour plots of $f = h(\sqrt{x^\top M x}) - r^\top x$ for different risk-weighting functions.}
    \label{fig:contour}
  \end{center}
\end{figure}

\section{Conclusions}\label{sec:concl}

We presented a branch-and-bound algorithm for a large class of convex
mixed-integer minimization problems arising in portfolio optimization.
Dual bounds are obtained by a modified version of the Frank-Wolfe
method. This is motivated mainly by
two reasons. On the one hand, the Frank-Wolfe algorithm, at each iteration, gives
a valid dual bound for the original mixed-integer problem, therefore
it may allow an early pruning of the node. On the other hand, the cost
per iteration is very low, since the computation of the descent
direction and the update of the objective function can be performed in
a very efficient way. Furthermore, the devised Frank-Wolfe method
benefits from the use of a non-monotone Armijo line search. Within the
branch-and-bound scheme, we propose different warmstarting
strategies. The branch-and-bound algorithm has been tested on a set of
real-world instances for the capital budgeting problem, considering
different classes of risk-weighting functions. Experimental results show that
the proposed approach significantly outperforms the MISOCP solver of
\texttt{CPLEX 12.6} for instances where a linear risk-weighting function is
considered.

\section{Appendix}\label{sec:appendix}
\appendix{
\par\noindent
Proof of Lemma~\ref{lemma1}:
\begin{proof}
  First note that the definition of~$\bar f^k$ ensures~$\bar
f^k\le f(x^0)$ and hence~$f(x^k)\le f(x^0)$ for all~$k$, which
proves~(i). For~(ii), we have that
\begin{equation*}
  \bar f^{k+1} = \max\limits_{0\le i\le \min\{p_{nm},k+1\}} f(x^{k + 1
    -i}) \le \max\{ \bar f^k , f(x^{k+1})\}.
\end{equation*}
Since $f(x^{k+1})<\bar f^k$ by the definition of the line search, we
derive $\bar f^{k+1} \le \bar f^k$, which proves that the sequence $\{
\bar f^k\}_{k\in\N}$ is non-increasing. By~(i), this sequence is
bounded from below by the minimum of~$f$ on $\mathcal L(x^0)\cap S$,
which exists by Lemma~\ref{basiclemma}, and hence converges.

\end{proof}

\par\medskip\noindent
Proof of Lemma~\ref{lemma2}:
\begin{proof}
 For each~$k\in\N$, choose $t^k\in\{k - \min(k,p_{nm}),\dots,k\}$ with $\bar f^k = f(x^{t^k})$.
We prove by induction that for any fixed integer $i\ge 0$ 
we have 
\begin{equation}\label{induction}
  \lim_{k\rightarrow \infty} f(x^{t^k - i}) = \lim_{k\rightarrow \infty} f(x^{t^k}) = \lim_{k\rightarrow \infty} \bar f^k = \bar f.
\end{equation}
Suppose at first $i = 0$. Then~\eqref{induction} follows from Lemma~\ref{lemma1}. 

We now assume that \eqref{induction} holds for $i\ge 0$ and we prove
that it holds for index $i+1$. We have
\begin{equation*}
  f(x^{t^k - i})  \le \bar f^{t^k -i-1} + \gamma_1  \alpha^{t^k -i -1} 
\nabla f(x^{t^k - i - 1})^\top d^{t^k - i - 1} - \gamma_2 (\alpha^{t^k -i-1})^2\|d^{t^k -i-1}\|^2,
\end{equation*}
so that the same reasoning as before yields
\begin{equation}\label{ind1}
  f(x^{t^k - i}) - \bar f^{t^k -i -1} \le - \gamma_2 (\alpha^{t^k -i-1})^2\|d^{t^k -i-1}\|^2.
\end{equation}
The left hand side of~\eqref{ind1} converges to zero since~\eqref{induction} holds for $i$ and the term $f(x^{t^k - i})$ converges 
to $\bar f$ (by the inductive hypothesis), 
as well as $\bar f^{t^k -i -1}$ because of Lemma~\ref{lemma1}
(and the fact that $k - (t^k - i -1)$ is bounded by $p_{nm} + i + 1$). Then,
$$\lim_{k\rightarrow \infty} (\alpha^{t^k -i - 1})^2 \|d^{t^k -i
  -1}\|^2 = 0,$$
so that
$\lim_{k\rightarrow \infty}  \| x^{t^k -i} - x^{t^k -i - 1} \| = 0$. Again, uniform continuity 
of $f(x)$ over $\mathcal L(x^0)\cap S$ yields~\eqref{induction} for index $i+1$.

To conclude the proof, let $T^k = t^{k + p_{nm} + 1}$ and note that for
any~$k$ we can write
\begin{equation*}
 f(x^k) = f(x^{T^k}) - \sum_{i=0}^{T^k - k -1}(f(x^{T^k - i}) -  f(x^{T^k - i -1})).
\end{equation*}
Therefore, 
since the summation vanishes and $f(x^{T^k})= \bar f^{k+p_{nm}+1}$ converges to $\bar f$ from 
Lemma~\ref{lemma1}, taking the limit for $k\rightarrow \infty$ and
observing~$T^k - k -1\le p_{nm}$ we obtain the result.

\end{proof}

\par\medskip\noindent
Proof of Lemma~\ref{prop:limgraddir}:
\begin{proof}
First note that $\nabla f(x^k)^\top d^k < 0$ for all~$k\in\N$.  Let
$\alpha^k$ be the stepsize used by \texttt{NM-MFW} at
iteration $k$. Then,
$$\bar f^k -  f( x^k+ \alpha^k d^k )\ge  \gamma_1\, \alpha^k\, |\nabla f(x^k)^{\top} d^k| + \gamma_2\,(\alpha^k)^2\,\|d^k\|^2 
\geq  \gamma_1\, \alpha^k\, |\nabla f(x^k)^{\top} d^k|\ge 0.$$
By Lemma~\ref{lemma2}, the left hand side converges to zero, hence
\begin{equation}\label{hello}
  \lim_{k\rightarrow \infty} \alpha^k\, |\nabla f(x^k)^{\top} d^k| = 0.
\end{equation}
Since~$f$ is continuously differentiable on the compact set~$\mathcal
L(x^0)\cap S$ by Lemma~\ref{basiclemma} and~$d^k$ is bounded on~$S$,
the sequence~$\nabla f(x^k)^\top d^k$ is bounded. It thus suffices to
show that any convergent subsequence of~$\nabla f(x^k)^\top d^k$
converges to zero.

We assume by contradiction that a subsequence exists with
\begin{equation*}
  \lim_{i\rightarrow \infty} \nabla f(x^{k_i})^{\top} d^{k_i} =
  -\eta < 0.
\end{equation*}
Since the sequences~$\{x^k\}_{k\in\N}$ and~$\{d^k\}_{k\in\N}$ are
bounded, we can switch to an appropriate subsequence and assume that
$\lim_{k\rightarrow \infty} x^k = \bar x$ and $\lim_{k\rightarrow
  \infty} d^k = \bar d$ exist. From~\eqref{hello} we obtain
\begin{equation}\label{alphazero}
  \lim_{k\rightarrow \infty}\alpha^k=0,
\end{equation}
and the continuity of the
gradient in $\mathcal L(x^0)\cap S$ implies
$$\nabla f(\bar x)^\top \bar d = \lim_{k\rightarrow \infty} \nabla f(x^k)^\top d^k = -\eta < 0.$$
Since $\alpha_{max}\ge\beta >0$ and the sequence $\alpha^k$ is converging to zero, a value $\bar k \in \N$ 
exists such that $\alpha^k< \alpha_{max}$, for $k\ge \bar k$.
In other words, for $k\ge \bar k$  the stepsize $\alpha^k$ cannot be set equal to the maximum stepsize 
and, taking into account the non-monotone Armijo line search, we can write
$$f\Big( x^k+ \frac{\alpha^k}{\delta} d^k \Big)> \bar f^k+\gamma_1\, \frac{\alpha^k}{\delta}\,
\nabla f(x^k)^{\top} d^k - \gamma_2\,\Big(\frac{\alpha^k}{\delta}\Big)^2\,\|d^k\|^2.$$
Hence, due to the fact that $\bar f^k \ge f(x^k)$, we get
\begin{equation}\label{alphadelta}
  f\Big( x^k+ \frac{\alpha^k}{\delta} d^k \Big) - f(x^k)> \gamma_1\, \frac{\alpha^k}{\delta}\, 
  \nabla f(x^k)^{\top} d^k - \gamma_2\,\Big(\frac{\alpha^k}{\delta}\Big)^2\,\|d^k\|^2. 
\end{equation}
Since $f$ is continuously differentiable in $\mathcal L(x^0)\cap S$, we can apply the Mean Value Theorem and we have that $s_k\in [0,1]$ exists such that
\begin{equation}\label{meanvaltheo}
  f\Big( x^k+ \frac{\alpha^k}{\delta} d^k \Big) = f(x^k) + \frac{\alpha^k}{\delta}\nabla f\Big( x^k+ s_k\frac{\alpha^k}{\delta} d^k \Big)^\top d^k.
\end{equation}
In particular, we have $\lim_{k\rightarrow \infty} x^k+
s_k\frac{\alpha^k}{\delta} d^k = \bar x$, by~\eqref{alphazero} and
since~$s_k$ and~$d^k$ are bounded.  By substituting~\eqref{meanvaltheo}
within~\eqref{alphadelta} we have
$$
\nabla f\Big( x^k+ s\frac{\alpha^k}{\delta} d^k \Big)^\top d^k > \gamma_1\, 
\nabla f(x^k)^{\top} d^k - \gamma_2\,\frac{\alpha^k}{\delta}\,\|d^k\|^2.
$$
Considering the limit on both sides we get
$$-\eta = \nabla f(\bar x)^\top \bar d > \gamma_1\, \nabla f(\bar x)^\top \bar d = - \gamma_1 \eta$$
which is a contradiction since $\gamma_1\in (0, \frac{1}{2})$ and $-\eta<0$.

\end{proof}

\par\medskip\noindent
Proof of Theorem~\ref{Theo:convergence}:
\begin{proof}
If \texttt{NM-MFW} does not stop in a finite number of
iterations at an optimal solution, from Lemma~\ref{prop:limgraddir} we
have that
$$
\lim_{k\rightarrow \infty} \nabla f(x^k)^\top d^k =0.
$$
Let~$\xstar$ be any limit point of~$\{x^k\}_{k\in\N}$. Since the
sequence $\{d^k\}_{k\in\N}$ is bounded, we can switch to an
appropriate subsequence and assume that
$$\lim_{k\rightarrow \infty} x^k = \xstar;\quad\lim_{k\rightarrow \infty} d^k = d^\star.$$
Therefore
$$\nabla f(\xstar)^\top d^\star = \lim_{k\rightarrow \infty} \nabla f(x^k)^\top d^k =0.$$
From the definition of $d^k$ (implied by~\eqref{dirchoice} and the definition of $d^{TS}$) we have
$$ \nabla f(x^k)^\top d^k \le \nabla f(x^k)^\top (x - x^k) \quad \forall\; x\in S.$$
Taking the limit for $k\rightarrow \infty$ yields
$$0 = \nabla f(\xstar)^\top d^\star \le  \nabla f(\xstar)^\top (x - \xstar)\quad \forall\; x\in S,$$
showing that $\xstar$ is an optimal solution for Problem~\eqref{ContRel0}.

\end{proof}
}

%

\end{document}